\documentclass[11pt]{article}

\usepackage{amssymb,amsmath,amsfonts,amsthm}
\usepackage{latexsym}
\usepackage{graphics}
\usepackage{indentfirst}

\newtheorem*{thm*}{Theorem}
\newtheorem{thm}{Theorem}
\newtheorem{lem}[thm]{Lemma}
\newtheorem{pro}[thm]{Proposition}

\newtheorem{ques}[thm]{Question}

\newcommand{\N}{\mathbb{N}}

\begin{document}

\title{Answers to Two Questions on the DP Color Function}

\author{Jeffrey A. Mudrock\footnotemark[1] and Seth Thomason\footnotemark[1]}

\footnotetext[1]{Department of Mathematics, College of Lake County, Grayslake, IL 60030.  E-mail:  {\tt {jmudrock@clcillinois.edu}}}

\maketitle

\begin{abstract}

DP-coloring is a generalization of list coloring that was introduced in 2015 by Dvo\v{r}\'{a}k and Postle.  The chromatic polynomial of a graph is a notion that has been extensively studied since the early 20th century.  The chromatic polynomial of graph $G$ is denoted $P(G,m)$, and it is equal to the number of proper $m$-colorings of $G$.  In 2019, Kaul and Mudrock introduced an analogue of the chromatic polynomial for DP-coloring; specifically, the DP color function of graph $G$ is denoted $P_{DP}(G,m)$.  Two fundamental questions posed by Kaul and Mudrock are: (1) For any graph $G$ with $n$ vertices, is it the case that $P(G,m)-P_{DP}(G,m) = O(m^{n-3})$ as $m \rightarrow \infty$? and (2) For every graph $G$, does there exist $p,N \in \N$ such that $P_{DP}(K_p \vee G, m) = P(K_p \vee G, m)$ whenever $m \geq N$?  We show that the answer to both these questions is yes.  In fact, we show the answer to (2) is yes even if we require $p=1$.

\medskip



\end{abstract}

\section{Introduction}\label{intro}

In this note all graphs are nonempty, finite, simple graphs unless otherwise noted.  Generally speaking we follow West~\cite{W01} for terminology and notation.  The set of natural numbers is $\N = \{1,2,3, \ldots \}$.  For $m \in \N$, we write $[m]$ for the set $\{1, \ldots, m \}$.  Given a set $A$, $\mathcal{P}(A)$ is the power set of $A$.  If $G$ is a graph and $S, U \subseteq V(G)$, we use $G[S]$ for the subgraph of $G$ induced by $S$, and we use $E_G(S, U)$ for the subset of $E(G)$ with one endpoint in $S$ and one endpoint in $U$.  If an edge in $E(G)$ connects the vertices $u$ and $v$, the edge can be represented by $uv$ or $vu$.  If $G$ and $H$ are vertex disjoint graphs, we write $G \vee H$ for the join of $G$ and $H$.  The \emph{cone of graph $G$} is $K_1 \vee G$.

\subsection{List Coloring and DP-Coloring} \label{basic}

In the classical vertex coloring problem we wish to color the vertices of a graph $G$ with up to $m$ colors from $[m]$ so that adjacent vertices receive different colors, a so-called \emph{proper $m$-coloring}. The chromatic number of a graph $G$, denoted $\chi(G)$, is the smallest $m$ such that $G$ has a proper $m$-coloring.  List coloring, a well-known variation on classical vertex coloring, was introduced independently by Vizing~\cite{V76} and Erd\H{o}s, Rubin, and Taylor~\cite{ET79} in the 1970s.  For list coloring, we associate a \emph{list assignment} $L$ with a graph $G$ such that each vertex $v \in V(G)$ is assigned a list of colors $L(v)$ (we say $L$ is a list assignment for $G$).  Then, $G$ is \emph{$L$-colorable} if there exists a proper coloring $f$ of $G$ such that $f(v) \in L(v)$ for each $v \in V(G)$ (we refer to $f$ as a \emph{proper $L$-coloring} of $G$).  A list assignment $L$ is called a \emph{$k$-assignment} for $G$ if $|L(v)|=k$ for each $v \in V(G)$.  The \emph{list chromatic number} of a graph $G$, denoted $\chi_\ell(G)$, is the smallest $k$ such that $G$ is $L$-colorable whenever $L$ is a $k$-assignment for $G$.  We say $G$ is \emph{$k$-choosable} if $k \geq \chi_\ell(G)$.  Since $G$ must be $L$-colorable whenever $L$ is a $\chi_\ell(G)$-assignment for $G$ that assigns the same list of colors to each element in $V(G)$, it is clear that $\chi(G) \leq \chi_\ell(G)$.  This inequality may be strict since it is known that there are bipartite graphs with arbitrarily large list chromatic number (see~\cite{ET79}).

In 2015, Dvo\v{r}\'{a}k and Postle~\cite{DP15} introduced a generalization of list coloring called DP-coloring (they called it correspondence coloring) in order to prove that every planar graph without cycles of lengths 4 to 8 is 3-choosable. DP-coloring has been extensively studied over the past 5 years (see e.g.,~\cite{B16,B17, BK17, BK182, BK18, KM19, KM20, KO18, KO182, LL19, LLYY19, Mo18, M18}). Intuitively, DP-coloring is a variation on list coloring where each vertex in the graph still gets a list of colors, but identification of which colors are different can change from edge to edge.  Following~\cite{BK17}, we now give the formal definition.  Suppose $G$ is a graph.  A \emph{cover} of $G$ is a pair $\mathcal{H} = (L,H)$ consisting of a graph $H$ and a function $L: V(G) \rightarrow \mathcal{P}(V(H))$ satisfying the following four requirements:

\vspace{5mm}

\noindent(1) the set $\{L(u) : u \in V(G) \}$ is a partition of $V(H)$; \\
(2) for every $u \in V(G)$, the graph $H[L(u)]$ is complete; \\
(3) if $E_H(L(u),L(v))$ is nonempty, then $u=v$ or $uv \in E(G)$; \\
(4) if $uv \in E(G)$, then $E_H(L(u),L(v))$ is a matching (the matching may be empty).

\vspace{5mm}

Suppose $\mathcal{H} = (L,H)$ is a cover of $G$.  We refer to the edges of $H$ connecting distinct parts of the partition $\{L(v) : v \in V(G) \}$ as \emph{cross-edges}. An \emph{$\mathcal{H}$-coloring} of $G$ is an independent set in $H$ of size $|V(G)|$.  It is immediately clear that an independent set $I \subseteq V(H)$ is an $\mathcal{H}$-coloring of $G$ if and only if $|I \cap L(u)|=1$ for each $u \in V(G)$.  We say $\mathcal{H}$ is \emph{$m$-fold} if $|L(u)|=m$ for each $u \in V(G)$.  The \emph{DP-chromatic number} of $G$, $\chi_{DP}(G)$, is the smallest $m \in \N$ such that $G$ has an $\mathcal{H}$-coloring whenever $\mathcal{H}$ is an $m$-fold cover of $G$.

Suppose $\mathcal{H} = (L,H)$ is an $m$-fold cover of $G$.  We say that $\mathcal{H}$ has a \emph{canonical labeling} if it is possible to name the vertices of $H$ so that $L(u) = \{ (u,j) : j \in [m] \}$ and $(u,j)(v,j) \in E(H)$ for each $j \in [m]$ whenever $uv \in E(G)$.~\footnote{When $\mathcal{H}=(L,H)$ has a canonical labeling, we will always refer to the vertices of $H$ using this naming scheme.}  Clearly, when $\mathcal{H}$ has a canonical labeling, $G$ has an $\mathcal{H}$-coloring if and only if $G$ has a proper $m$-coloring.  Also, given an $m$-assignment, $L$, for a graph $G$, it is easy to construct an $m$-fold cover $\mathcal{H}'$ of $G$ such that $G$ has an $\mathcal{H}'$-coloring if and only if $G$ has a proper $L$-coloring (see~\cite{BK17}).  It follows that $\chi(G) \leq \chi_\ell(G) \leq \chi_{DP}(G)$.  The second inequality may be strict since it is easy to prove that $\chi_{DP}(C_n) = 3$ whenever $n \geq 3$, but the list chromatic number of any even cycle is 2 (see~\cite{BK17} and~\cite{ET79}).

In some instances DP-coloring behaves similar to list coloring, but there are some interesting differences.  Molloy~\cite{Mo18} has shown that Kahn's~\cite{K96} result that the list edge-chromatic number of a simple graph asymptotically equals the edge-chromatic number holds for DP-coloring as well. Thomassen~\cite{T94} famously proved that every planar graph is 5-choosable, and Dvo\v{r}\'{a}k and Postle~\cite{DP15} observed that the DP-chromatic number of every planar graph is at most 5.  Also, Molloy~\cite{M17} recently improved a theorem of Johansson by showing that every triangle-free graph $G$ with maximum degree $\Delta(G)$ satisfies $\chi_\ell(G) \leq (1 + o(1)) \Delta(G)/ \log(\Delta(G))$.  Bernshteyn~\cite{B17} subsequently showed that this bound also holds for the DP-chromatic number.  On the other hand, Bernshteyn~\cite{B16} showed that if the average degree of a graph $G$ is $d$, then $\chi_{DP}(G) = \Omega(d/ \log(d))$.  This is in striking contrast to the celebrated result of Alon~\cite{A00} that says $\chi_\ell(G) = \Omega(\log(d))$.  It was also recently shown in~\cite{BK17} that there exist planar bipartite graphs with DP-chromatic number 4 even though the list chromatic number of any planar bipartite graph is at most 3~\cite{AT92}.  A famous result of Galvin~\cite{G95} says that if $G$ is a bipartite multigraph and $L(G)$ is the line graph of $G$, then $\chi_\ell(L(G)) = \chi(L(G)) = \Delta(G)$.  However, it is also shown in~\cite{BK17} that every $d$-regular graph $G$ satisfies $\chi_{DP}(L(G)) \geq d+1$.
 
\subsection{Counting Proper Colorings, List Colorings, and DP-Colorings}

In 1912 Birkhoff introduced the notion of the chromatic polynomial in hopes of using it to make progress on the four color problem.  For $m \in \N$, the \emph{chromatic polynomial} of a graph $G$, $P(G,m)$, is the number of proper $m$-colorings of $G$.  It can be shown that $P(G,m)$ is a polynomial in $m$ of degree $|V(G)|$ (see~\cite{B12}).  For example, $P(K_n,m) = \prod_{i=0}^{n-1} (m-i)$, $P(C_n,m) = (m-1)^n + (-1)^n (m-1)$, $P(T,m) = m(m-1)^{n-1}$ whenever $T$ is a tree on $n$ vertices, and $P(K_1 \vee G, m) = m P(G, m-1)$ (see~\cite{W01}).

The notion of chromatic polynomial was extended to list coloring in the 1990s.   In particular, if $L$ is a list assignment for $G$, we use $P(G,L)$ to denote the number of proper $L$-colorings of $G$. The \emph{list color function} $P_\ell(G,m)$ is the minimum value of $P(G,L)$ where the minimum is taken over all possible $m$-assignments $L$ for $G$.  It is clear that $P_\ell(G,m) \leq P(G,m)$ for each $m \in \N$ since we must consider the $m$-assignment that assigns the same $m$ colors to all the vertices in $G$ when considering all possible $m$-assignments for $G$.  In general, the list color function can differ significantly from the chromatic polynomial for small values of $m$.  However, for large values of $m$, Wang, Qian, and Yan~\cite{WQ17} (improving upon results in~\cite{D92} and~\cite{T09}) showed the following in 2017.

\begin{thm} [\cite{WQ17}] \label{thm: WQ17}
If $G$ is a connected graph with $l$ edges, then $P_{\ell}(G,m)=P(G,m)$ whenever $m > \frac{l-1}{\ln(1+ \sqrt{2})}$.
\end{thm}

It is also known that $P_{\ell}(G,m)=P(G,m)$ for all $m \in \N$ when $G$ is a cycle or chordal (see~\cite{KN16} and~\cite{AS90}).  Moreover, if $P_{\ell}(G,m)=P(G,m)$ for all $m \in \N$, then $P_{\ell}(K_n \vee G,m)=P(K_n \vee G,m)$ for each $n, m \in \N$ (see~\cite{KM18}). See~\cite{T09} for a survey of known results and open questions on the list color function.

In 2019, Kaul and the first author introduced a DP-coloring analogue of the chromatic polynomial in hopes of gaining a better understanding DP-coloring and using it as a tool for making progress on some open questions related to the list color function~\cite{KM19}.  Specifically, suppose $\mathcal{H} = (L,H)$ is a cover of graph $G$.  Let $P_{DP}(G, \mathcal{H})$ be the number of $\mathcal{H}$-colorings of $G$.  Then, the \emph{DP color function} of $G$, $P_{DP}(G,m)$, is the minimum value of $P_{DP}(G, \mathcal{H})$ where the minimum is taken over all possible $m$-fold covers $\mathcal{H}$ of $G$.~\footnote{We take $\N$ to be the domain of the DP color function of any graph.}  It is easy to show that for any graph $G$ and $m \in \N$, $P_{DP}(G, m) \leq P_\ell(G,m) \leq P(G,m)$.~\footnote{To prove this, recall that for any $m$-assignment $L$ for $G$, an $m$-fold cover $\mathcal{H}'$ of $G$ such that $G$ has an $\mathcal{H}'$-coloring if and only if $G$ has a proper $L$-coloring is constructed in~\cite{BK17}.  It is easy to see from the construction in~\cite{BK17} that there is a bijection between the proper $L$-colorings of $G$ and the $\mathcal{H}'$-colorings of $G$.}  Note that if $G$ is a disconnected graph with components: $H_1, H_2, \ldots, H_t$, then $P_{DP}(G, m) = \prod_{i=1}^t P_{DP}(H_i,m)$.  So, we will only consider connected graphs from this point forward unless otherwise noted.

As with list coloring and DP-coloring, the list color function and DP color function of certain graphs behave similarly.  However, for some graphs there are surprising differences.  For example, similar to the list color function,  $P_{DP}(G,m) = P(G,m)$ for every $m \in \N$  whenever $G$ is chordal or an odd cycle~\cite{KM19}.  On the other hand, we have the following two results.

\begin{thm} [\cite{KM19}] \label{thm: evengirth}
If $G$ is a graph with girth that is even, then there is an $N \in \N$ such that $P_{DP}(G,m) < P(G,m)$ whenever $m \geq N$.  Furthermore, for any integer $g \geq 3$ there exists a graph $H$ with girth $g$ and an $N \in \N$ such that $P_{DP}(H,m) < P(H,m)$ whenever $m \geq N$. 
\end{thm}

This result is particularly surprising since Theorem~\ref{thm: WQ17} implies that the list color function of any graph eventually equals its chromatic polynomial.  The following is also known.

\begin{thm} [\cite{KM19}] \label{thm: asymptotic}
For any graph $G$ with $n$ vertices,
$$P(G,m)-P_{DP}(G,m) = O(m^{n-2}) \; \; \text{as $m \rightarrow \infty$.}$$
\end{thm}

In studying the tightness of Theorem~\ref{thm: asymptotic}, the authors of~\cite{KM19} mentioned that if $G$ is a unicyclic graph~\footnote{A \emph{unicyclic graph} is a connected graph containing exactly one cycle.} on $n$ vertices that contains a cycle of length 4, then $P(G,m)-P_{DP}(G,m) = \Theta(m^{n-3})$.  However, they stated that ``we do not have an example of a graph $G$ such that $P(G,m)-P_{DP}(G,m) = \Theta(m^{n-2})$."  Motivated by a result of Bernshteyn, Kostochka, and Zhu~\cite{BK18} that says for any graph $G$ there exists an $N \leq 3|E(G)|$ such that $\chi_{DP}(K_p \vee G) = \chi(K_p \vee G)$ whenever $p \geq N$, the authors of~\cite{KM19} also studied $P_{DP}(K_p \vee G, m)$. Interestingly, it turns out that the question of whether there exist $p,N \in \N$ such that $P_{DP}(K_p \vee G, m) = P(K_p \vee G , m)$ whenever $m \geq N$ is related to the asymptotics of $P(G,m)-P_{DP}(G,m)$. In fact, the following two questions were both posed in~\cite{KM19}.  These two questions are the focus of this note.  

\begin{ques} \label{ques: asymptotic}
For any graph $G$ with $n$ vertices, is it the case that $P(G,m)-P_{DP}(G,m) = O(m^{n-3})$ as $m \rightarrow \infty$?
\end{ques}

\begin{ques} \label{ques: join}
For every graph $G$, does there exist $p,N \in \N$ such that $P_{DP}(G \vee K_p, m) = P(G \vee K_p , m)$ whenever $m \geq N$?
\end{ques}

In~\cite{KM19} it is shown that if the the answer to Question~\ref{ques: asymptotic} is yes, then the answer to Question~\ref{ques: join} must be yes.  We will show that the answer to Question~\ref{ques: asymptotic} is yes.  This of course implies that the answer to Question~\ref{ques: join} is yes, but we will show that its answer is yes even when $p$ is fixed to 1.

\subsection{Summary of Results}

We begin by showing the following.

\begin{thm} \label{thm: general}
Suppose $g$ is an odd integer with $g \geq 3$.  If $G$ is a graph on $n$ vertices with girth $g$ or $g+1$, then $P(G,m) - P_{DP}(G,m) = O(m^{n-g})$ as $m \rightarrow \infty$.  Consequently, $P(M,m) - P_{DP}(M,m) = O(m^{|V(M)|-3})$ as $m \rightarrow \infty$ for any graph $M$. 
\end{thm}

When considering the third sentence of Theorem~\ref{thm: general}, recall that if the girth of a graph is infinite, then the graph is acyclic and therefore chordal which means the DP color function of the graph is always equal to its chromatic polynomial.  A result in~\cite{KM19} implies that if $G$ is a unicyclic graph on $n$ vertices with girth $2k+2$ where $k \in \N$, then $P_{DP}(G,m) = \Theta(m^{n-2k-1})$.  It can also be shown that for any odd integer $g$ with $g \geq 3$, if $G$ consists of a cycle on $g$ vertices and a cycle on $g+1$ vertices such that the cycles share exactly one vertex, then $P(G,m) - P_{DP}(G,m) = \Theta(m^{(2g)-g})$.  This demonstrates the tightness of Theorem~\ref{thm: general} for all possible girths. 

We end this note by proving the following.

\begin{thm} \label{thm: cone}
For any graph $G$, there is an $N \in \N$ such that $P_{DP}(K_1 \vee G,m) = P(K_1 \vee G,m)$ whenever $m \geq N$.   
\end{thm}

Theorem~\ref{thm: cone} shows that the DP color function of $K_1 \vee G$ behaves like the list color function of $K_1 \vee G$ since the DP color function of $K_1 \vee G$ eventually equals the chromatic polynomial of $K_1 \vee G$.  It is worth mentioning that in this note no attempt has been made to minimize the value of $N$ in Theorem~\ref{thm: cone}.  It would be interesting to study the threshold at which $P_{DP}(K_1 \vee G,m) = P(K_1 \vee G,m)$ for a given graph $G$.

\section{Proofs of Results} \label{main}

The key to proving our results is generalizing the proof technique of the following classical result to the context of DP-coloring.

\begin{thm} [\cite{W32}] \label{pro: base}
Suppose $G$ is a graph.  Then,
$$P(G,m) = \sum_{A \subseteq E(G)} (-1)^{|A|} m^{k_A}$$
where $k_A$ is the number of components of the spanning subgraph of $G$ with edge set $A$.
\end{thm}

The next four results will also be useful tools to keep in mind.

\begin{pro} [\cite{W32}] \label{pro: coefficients2}
Suppose $G$ is a graph on $n$ vertices.  Then there are nonnegative integers $a_0, \ldots, a_n$ such that $P(G,m) = \sum_{i=0}^{n} (-1)^i a_i m^{n-i}$.  Furthermore, if $G$ has $c$ components, then $a_0, \ldots, a_{n-c}$ are all positive integers, and $a_{n-c+1} = \cdots = a_n = 0$.  
\end{pro}  

\begin{pro} [\cite{W32}] \label{pro: coefficients}
Suppose $G$ is a graph with $s$ edges and $n$ vertices having girth $g \in \N$.  Suppose $P(G,m) = \sum_{i=0}^{n} (-1)^i a_i m^{n-i}$.  Then, for $i = 0, 1, \ldots, g-2$
$$ a_i = \binom{s}{i} \; \; \text{and} \; \; a_{g-1} = \binom{s}{g-1} - t$$
where $t$ is the number of cycles of length $g$ contained in $G$.
\end{pro}

\begin{pro} [\cite{KM19}] \label{pro: tree}
Suppose $T$ is a tree and $\mathcal{H} = (L,H)$ is an $m$-fold cover of $T$ such that $E_H(L(u),L(v))$ is a perfect matching whenever $uv \in E(T)$.  Then, $\mathcal{H}$ has a canonical labeling.
\end{pro}

\begin{pro} \label{pro: obvious}
Suppose that $\mathcal{H} = (L,H)$ is an $m$-fold cover of graph $G$ and $\mathcal{H}$ has a canonical labeling.  Let $B_i = \{(v,i): v \in V(G) \}$ for each $i \in [m]$. Then, $I \subset V(H)$ satisfies: $|I \cap L(v)|=1$ for each $v \in V(G)$ and $H[I]$ is isomorphic to $G$ if and only if $I = B_j$ for some $j \in [m]$.
\end{pro}  

\begin{proof}
For each $i \in [m]$, it is clear that $|B_i \cap L(v)|=1$ for each $v \in V(G)$ and $H[B_i]$ is isomorphic to $G$.  Conversely, suppose that and $I \notin \{B_1, \ldots, B_m \}$.   Since  $|I \cap L(v)|=1$ for each $v \in V(G)$, $H[I]$ has fewer edges than $G$ contradicting the fact that $H[I]$ is isomorphic to $G$.
\end{proof}

\subsection{Proof of Theorem~\ref{thm: general}} \label{general}

We will now introduce some notation that will be used for the remainder of this note.  Suppose that $G$ is a graph on $n \geq 3$ vertices with $|E(G)| \geq 3$.  Let $s = |E(G)|$, and $E(G) = \{e_1, \ldots, e_s \}$.  Also, for some $m \in \N$ suppose that $\mathcal{H}= (L,H)$ is an $m$-fold cover of $G$ satisfying $|E_H(L(u), L(v))| = m$ whenever $uv \in E(G)$.  

Let $\mathcal{U} = \{ I \subseteq V(H) : |L(v) \cap I| = 1 \text{ for each } v \in V(G) \}$.  Clearly, $|\mathcal{U}| = m^n$.  Now, for each $i \in [s]$, suppose $e_i=u_i v_i$, and let $S_i$ be the set consisting of each $I \in \mathcal{U}$ with the property that $H[I]$ contains an edge in $E_H(L(u_i), L(v_i))$.  Also, for each $i \in [s]$ let $C_i = \mathcal{U} - S_i$.  Clearly,
$$P_{DP}(G, \mathcal{H}) = \left | \bigcap_{i=1}^s C_i \right |.$$
So, by the Inclusion-Exclusion Principle, we see that
$$P_{DP}(G, \mathcal{H}) = |\mathcal{U}| - \left | \bigcup_{i=1}^s S_i \right | =  m^n - \sum_{k=1}^s (-1)^{k-1} \left ( \sum_{1 \leq i_1 < \cdots < i_k \leq s} \left | \bigcap_{j=1}^k S_{i_j} \right| \right).$$  
The following Lemma is the key to our proof of Theorem~\ref{thm: general}.

\begin{lem} \label{lem: formulas2}
Assuming the set up established above, suppose that $G$ is a graph of girth $g \in \N$.  Then, the following three statements hold. \\
(i) For any $k \in [g-1]$ and $i_1, \ldots, i_k \in [s]$ satisfying $i_1 < \cdots < i_k$, $\left | \bigcap_{j=1}^k S_{i_j} \right|  = m^{n-k}$. \\  
(ii) If $e_{i_1}, \ldots, e_{i_g}$ are distinct edges in $G$, then $\left | \bigcap_{j=1}^g S_{i_j} \right| \leq m^{n-g+1}$.  Moreover, $\left | \bigcap_{j=1}^g S_{i_j} \right| = m^{n-g}$ when $e_{i_1}, \ldots, e_{i_g}$ are not the edges of a $g$-cycle in $G$. \\
(iii) For any $k \geq g+1$ and $i_1, \ldots, i_k \in [s]$ satisfying $i_1 < \cdots < i_k$, $\left | \bigcap_{j=1}^k S_{i_j} \right|  \leq m^{n-g}$.  
\end{lem}

\begin{proof}
For Statement (i), suppose that $G'$ is the spanning subgraph of $G$ with $E(G')= \{e_{i_1}, \ldots, e_{i_k} \}$.  Since $G$ has girth $g$ and $k \in [g-1]$, $G'$ is an acyclic graph with $n-k$ components.  Suppose the components of $G'$ are $W_1, \ldots, W_{n-k}$ (each component is a tree).  Note that we can construct each element $I$ of $\bigcap_{j=1}^k S_{i_j}$ in $(n-k)$ steps as follows.  For each $i \in [n-k]$ consider the component $W_i$.  Suppose $V(W_i) = \{w_1, \ldots, w_l\}$.  Choose one element from each of $L(w_1), \ldots, L(w_l)$ so that the subgraph of $H$ induced by the set containing these chosen elements is isomorphic to $W_i$.  Then, place these chosen elements in $I$.   By Propositions~\ref{pro: tree} and~\ref{pro: obvious}, this step can be done in $m$ ways (regardless of the choices made in previous steps).  So, $\left | \bigcap_{j=1}^k S_{i_j} \right|  = m^{n-k}$.

For Statement (ii), the first part follows from Statement~(i) since $\left | \bigcap_{j=1}^g S_{i_j} \right| \leq \left | \bigcap_{j=1}^{g-1} S_{i_j} \right| = m^{n-g+1}$.  So, suppose that $e_{i_1}, \ldots, e_{i_g}$ are not the edges of a $g$-cycle in $G$.  Let $G''$ be the spanning subgraph of $G$ with $E(G'')= \{e_{i_1}, \ldots, e_{i_g}\}$.  Clearly, $G''$ is an acyclic graph with $n-g$ components. We can obtain $\left | \bigcap_{j=1}^g S_{i_j} \right| = m^{n-g}$ by using an argument similar to the argument used for the proof of Statement~(i). 

For Statement (iii), notice that we can assume without loss of generality that $e_{i_1}, \ldots, e_{i_g}$ are not the edges of a $g$-cycle in $G$.  So, by Statement~(ii), we see that $\left| \bigcap_{j=1}^k S_{i_j} \right| \leq \left | \bigcap_{j=1}^g S_{i_j} \right| = m^{n-g}.$ 
\end{proof}

We are now ready to prove Theorem~\ref{thm: general}.

\begin{proof}
Suppose $s=|E(G)|$ and $t$ is the number of $g$-cycles in $G$ (note that $t=0$ in the case that $G$ has girth $g+1$).  Since $g$ is odd, Propositions~\ref{pro: coefficients2} and~\ref{pro: coefficients} tell us that there is an $N \in \N$ such that 
$$P(G,m) \leq \left(\binom{s}{g-1} - t \right) m^{n-g+1} + \sum_{i=0}^{g-2} (-1)^i \binom{s}{i} m^{n-i}$$
whenever $m \geq N$.  Suppose that $m$ is a fixed natural number satisfying $m \geq N$.

Suppose that $\mathcal{H}= (L,H)$ is an $m$-fold cover of $G$ satisfying $P_{DP}(G, \mathcal{H}) = P_{DP}(G,m)$.  Clearly, we may assume that $|E_H(L(u), L(v))| = m$ whenever $uv \in E(G)$.  Now, assume we use the same notation described at the start of this Subsection.  By Statement~(i) of Lemma~\ref{lem: formulas2}, we have that
\begin{align*}
&P_{DP}(G,m) \\
&=  m^n + \sum_{k=1}^s (-1)^{k} \left ( \sum_{1 \leq i_1 < \cdots < i_k \leq s} \left | \bigcap_{j=1}^k S_{i_j} \right| \right) \\
&= \sum_{i=0}^{g-1} (-1)^i \binom{s}{i} m^{n-i} - \sum_{1 \leq i_1 < \cdots < i_g \leq s} \left | \bigcap_{j=1}^g S_{i_j} \right| + \sum_{k=g+1}^s (-1)^{k} \left ( \sum_{1 \leq i_1 < \cdots < i_k \leq s} \left | \bigcap_{j=1}^k S_{i_j} \right| \right).  
\end{align*}
We see that Statement~(ii) of Lemma~\ref{lem: formulas2} implies that
$$\sum_{1 \leq i_1 < \cdots < i_g \leq s} \left | \bigcap_{j=1}^g S_{i_j} \right| \leq tm^{n-g+1} + \left(\binom{s}{g} - t \right) m^{n-g}.$$
Furthermore, Statement~(iii) of Lemma~\ref{lem: formulas2} implies that
$$\sum_{k=g+1}^s (-1)^{k} \left ( \sum_{1 \leq i_1 < \cdots < i_k \leq s} \left | \bigcap_{j=1}^k S_{i_j} \right| \right) \geq -2^s m^{n-g}.$$
These facts imply that
$$P_{DP}(G,m) \geq  \left(\binom{s}{g-1} - t \right) m^{n-g+1} + \sum_{i=0}^{g-2} (-1)^i \binom{s}{i} m^{n-i} - \left(\binom{s}{g} - t \right) m^{n-g} - 2^s m^{n-g}.$$
So, we see that
$$P(G,m) - P_{DP}(G,m) \leq \left(\binom{s}{g} - t + 2^s \right) m^{n-g}.$$
The desired result immediately follows. 
\end{proof}

\subsection{Proof of Theorem~\ref{thm: cone}}

Notice that the result of Theorem~\ref{thm: cone} is obvious when $G$ is acyclic since the cone of such a graph is chordal.  So, throughout this Subsection suppose that $G$ is a graph with $n-1$ vertices where $n \geq 4$ and $s \geq 3$ edges.  Suppose that $E(G) = \{e_1, \ldots, e_s \}$.  Also, suppose that $M = K_1 \vee G$, and $w$ is the vertex corresponding to the copy of $K_1$ used to form $M$.  We use $e_{s+1}, \ldots, e_{s+n}$ to denote the edges in $E(M)$ that have $w$ as an endpoint.  We want to show that $P_{DP}(M,m) = P(M,m)$ for sufficiently large $m$, or equivalently, $P_{DP}(M,m) \geq P(M,m)$ for sufficiently large $m$.  

Since $M$ is a graph with $n$ vertices and $n+s$ edges, Propositions~\ref{pro: coefficients2} and~\ref{pro: coefficients} tell us that
$$P(M,m)= m^n - (n+s)m^{n-1} +  \left(\binom{n+s}{2} - t \right)m^{n-2} - a_3 m^{n-3} + O(m^{n-4}) \text{ as } m \rightarrow \infty$$
where $t$ is the number of 3-cycles contained in $M$ (note that $t \geq s$).  We now give a formula for $a_3$.  Let $A = \{A_1, \ldots, A_q \}$ be the set of spanning subgraphs of $M$ with $(n-3)$ components.  For each $i \in [q]$, it is straightforward to verify that $3 \leq |E(A_i)| \leq 6$.  For $i \in \{3, 4, 5, 6 \}$, let $P_i = \{E(A_j) : A_j \in A, |E(A_j)|=i \}$.  By Theorem~\ref{pro: base}, $a_3 = |P_3|-|P_4|+|P_5|-|P_6|$.

In this Subsection we are interested in finding a lower bound for $P_{DP}(M,m)$.  So, whenever $\mathcal{H} = (L,H)$ is an $m$-fold cover for $M$, we will assume that $|E_H(L(u),L(v))|=m$ for each $uv \in E(M)$.  We will also suppose without loss of generality that $L(u) = \{(u,j): j \in [m] \}$ for each $u \in V(M)$, and $(w,j)(v,j) \in E(H)$ for each $v \in V(G)$ and $j \in [m]$ (this is permissible by Proposition~\ref{pro: tree} since the spanning subgraph of $M$ with edge set $\{e_{s+1}, \ldots, e_{s+n}\}$ is a tree).  Also, for each $e_i \in E(G)$ we suppose that $e_i = u_iv_i$, and we let $x_{i, \mathcal{H}}$~\footnote{We will just write $x_i$ when $\mathcal{H}$ is clear from context.} be the number of edges in $E_H(L(u_i),L(v_i))$ that connect endpoints with differing second coordinates.  Finally, we let $x_\mathcal{H} = \sum_{i=1}^s x_{i, \mathcal{H}}$.  Clearly, if $x_{\mathcal{H}}=0$, then $\mathcal{H}$ has a canonical labeling and $P_{DP}(M, \mathcal{H})=P(M,m)$.  Also, $x_\mathcal{H}$ is the number of cross edges in $H$ that connect vertices with differing second coordinates.

For the next Lemma assume that $\mathcal{H} = (L,H)$ is an $m$-fold cover for $M$, and assume we are using the same notation as the beginning of Subsection~\ref{general} (with $M$ playing the role of $G$).

\begin{lem} \label{lem: three}
The following statments hold. \\
(i)  $\sum_{1 \leq i_1 < \cdots < i_3 \leq n+s} \left | \bigcap_{j=1}^3 S_{i_j} \right| \leq tm^{n-2} - x_\mathcal{H}m^{n-3} + |P_3|m^{n-3}$, \\
(ii) $\sum_{1 \leq i_1 < \cdots < i_4 \leq n+s} \left | \bigcap_{j=1}^4 S_{i_j} \right| \geq |P_4|m^{n-3} - 2|P_4|x_\mathcal{H}m^{n-4}$, \\ 
(iii) $\sum_{1 \leq i_1 < \cdots < i_5 \leq n+s} \left | \bigcap_{j=1}^5 S_{i_j} \right| \leq |P_5|m^{n-3} + \left(\binom{n+s}{5} - |P_5| \right) m^{n-4}$, \\
(iv) $\sum_{1 \leq i_1 < \cdots < i_6 \leq n+s} \left | \bigcap_{j=1}^6 S_{i_j} \right| \geq |P_6|m^{n-3}- 2|P_6|x_\mathcal{H}m^{n-4}$, and \\
(v) For $k \geq 7$, $\sum_{1 \leq i_1 < \cdots < i_k \leq n+s} \left | \bigcap_{j=1}^k S_{i_j} \right| \leq \binom{n+s}{k} m^{n-4}$.  
\end{lem}

\begin{proof}
For Statement~(i), suppose $x$, $y$, and $z$ are distinct edges in $E(M)$.  Let $M'$ be the spanning subgraph of $M$ with $E(M') = \{x,y,z\}$.  If $x$, $y$, and $z$ form a 3-cycle in $M$ containing $w$, then $M'$ consists of this 3-cycle and $n-3$ isolated vertices.  Notice that each 3-cycle in $M$ containing $w$ contains exactly one edge in $E(G)$.  So, we suppose that $z = e_i$ for some $i \in [s]$, then it is clear that $|S_x \cap S_y \cap S_z| = m^{n-3}(m-x_i) = m^{n-2} - x_im^{n-3}$.  In the case that $x$, $y$, and $z$ form a 3-cycle in $M$ not containing $w$, then Lemma~\ref{lem: formulas2} implies that $|S_x \cap S_y \cap S_z| \leq m^{n-2}$.  Finally, in the case that $x$, $y$, and $z$ do not form a 3-cycle in $M$ (note that there are $|P_3|$ such sets of three edges), Lemma~\ref{lem: formulas2} implies $|S_x \cap S_y \cap S_z| = m^{n-3}$.  Statement~(i) now follows immediately from these facts. 

For Statement~(ii), suppose $a,x,y$, and $z$ are distinct edges in $E(M)$.  Let $M'$ be the spanning subgraph of $M$ with $E(M') = \{a,x,y,z\}$.  If $M'$ contains a cycle, then $M'$ contains one cycle and consists of $n-3$ components (note that there are $|P_4|$ sets of four edges for which this happens).  Suppose that the components of $M'$ are $W_1, \ldots, W_{n-3}$, and assume that $W_1$ is the component of $M'$ containing the cycle.  Also, suppose that $V(W_1) = \{w_1, \ldots, w_l \}$.  Now, let $B_i = \{(w_j,i) : j \in [l] \}$ for each $i \in [m]$.  If $H[B_i]$ is not isomorphic to $W_1$, then one of the elements in $B_i$ must be the endpoint of a cross edge in $H$ that connects vertices with differing second coordinates. Let $\mathcal{B}$ consist of each $B_i \in \{B_1, \ldots, B_m\}$ with the property that $H[B_i]$ is not isomorphic to $W_1$.  Notice this means that for each $B_j \in \{B_1, \ldots, B_m\}-\mathcal{B}$, $H[B_j]$ is isomorphic to $W_1$, and there are are least $|\{B_1, \ldots, B_m\}-\mathcal{B}|$ ways to select one element from each of $L(w_1), \ldots, L(w_l)$ so that the subgraph of $H$ induced by the set containing these chosen elements is isomorphic to $W_1$.  Let $\mathcal{E}$ be the set of cross edges in $H$ that connect vertices with differing second coordinates (note that $|\mathcal{E}|=x_{\mathcal{H}}$).  We can construct a function $\eta: \mathcal{B} \rightarrow \mathcal{E}$ that maps each $B_i \in \mathcal{B}$ to one of the edges in $\mathcal{E}$ that has an endpoint in $B_i$.  Furthermore, if $B_i$, $B_j$, and $B_t$ are distinct elements of $\mathcal{B}$, then it is not possible for $\eta(B_i) = \eta(B_j) = \eta(B_t)$ since an edge only has two endpoints.  Consequently, $|\{B_1, \ldots, B_m\}-\mathcal{B}| \geq (m-2x_\mathcal{H})$.  So, $|S_a \cap S_x \cap S_y \cap S_z| \geq m^{n-4}(m-2x_\mathcal{H}) = m^{n-3} - 2x_\mathcal{H}m^{n-4}$.  Statement~(ii) now immediately follows. 

For Statement~(iii), suppose $a, b, x, y$, and $z$ are distinct edges in $E(M)$.  Let $M'$ be the spanning subgraph of $M$ with $E(M') = \{a,b,x,y,z\}$.  Suppose $M'$ consists of $n-3$ components (note that there are $|P_5|$ sets of five edges for which this happens).  Suppose that the components of $M'$ are $W_1, \ldots, W_{n-3}$.  Note that we can construct each element $I$ of $(S_a \cap S_b \cap S_x \cap S_y \cap S_z)$ in $(n-3)$ steps as follows.  For each $i \in [n-3]$ consider the component $W_i$.  If $\{a,b,x,y,z \} \cap E(W_i) \neq \emptyset$, then $V(W_i)$ has at least 2 elements, say $V(W_i) = \{w_1, \ldots, w_l\}$, choose one element from each of $L(w_1), \ldots, L(w_l)$ so that the subgraph of $H$ induced by these chosen elements is isomorphic to $W_i$ (this can be done in at most $m$ ways~\footnote{To see why this is so, consider a spanning tree of $W_i$ and apply Propositions~\ref{pro: tree} and~\ref{pro: obvious}.}).  Then, place these chosen elements in $I$.  If $\{a,b,x,y,z \} \cap E(W_i) = \emptyset$, then $W_i$ is a single vertex, say $V(W_i) = \{w\}$, and we choose an element of $L(w)$ to place in $I$.  Notice that in either case there are at most $m$ ways to complete the step.  Consequently, $|S_a \cap S_b \cap S_x \cap S_y \cap S_z| \leq m^{n-3}$.  A similar argument shows that when $M'$ has fewer than $n-3$ components, $|S_a \cap S_b \cap S_x \cap S_y \cap S_z| \leq m^{n-4}$.  Statement~(iii) now follows from the fact that $\sum_{1 \leq i_1 < \cdots < i_5 \leq n+s} \left | \bigcap_{j=1}^5 S_{i_j} \right|$ has $\binom{n+s}{5}$ terms.   

For Statement~(iv), suppose $a,b,c,x,y$, and $z$ are distinct edges in $E(M)$.  Let $M'$ be the spanning subgraph of $M$ with $E(M') = \{a,b,c,x,y,z\}$.  Suppose $M'$ consists of $n-3$ components (note that there are $|P_6|$ sets of six edges for which this happens).  It is easy to see that $M'$ must consist of a complete graph on four vertices and $n-4$ isolated vertices.  Suppose that the components of $M'$ are $W_1, \ldots, W_{n-3}$, and assume that $W_1 = K_4$.  Using an argument similar to the argument used to prove Statement~(ii), we obtain $|S_a \cap S_b \cap S_c \cap S_x \cap S_y \cap S_z| \geq m^{n-4}(m-2x_\mathcal{H}) = m^{n-3} - 2x_\mathcal{H}m^{n-4}$.  Statement~(iv) now immediately follows. 

For Statement~(v), suppose $k \geq 7$ and $1 \leq i_1 < \cdots < i_k \leq n+s$.  Let $M'$ be the spanning subgraph of $M$ with $E(M') = \{e_{i_1}, \ldots, e_{i_{k}}\}$.  Then, $M'$ must consist of at most $n-4$ components.  An argument similar to the argument used to prove Statement~(iii) then yields $\left | \bigcap_{j=1}^k S_{i_j} \right| \leq m^{n-4}$.  Statement~(v) now immediately follows.   
\end{proof}

We need one more Lemma before proving Theorem~\ref{thm: cone}.

\begin{lem} \label{lem: lower}
Suppose that $m \geq 2(|P_4|+|P_6|)$, and $\mathcal{H} = (L,H)$ is an $m$-fold cover for $M$ with $x_\mathcal{H} > 0$.  Then,
\begin{align*}
&P_{DP}(M, \mathcal{H}) \\
&\geq m^n - (n+s)m^{n-1} +  \left(\binom{n+s}{2} - t \right)m^{n-2} - a_3 m^{n-3} \\
&+ m^{n-3} - (2(|P_4|+|P_6|+2^{s-1}))m^{n-4}.
\end{align*}
\end{lem}

\begin{proof}
Using the notation established in Subsection~\ref{general} (with $M$ playing the role of $G$) along with Lemma~\ref{lem: formulas2}, we know that 
\begin{align*}
&P_{DP}(M,\mathcal{H}) \\
&=  m^n + \sum_{k=1}^{n+s} (-1)^{k} \left ( \sum_{1 \leq i_1 < \cdots < i_k \leq n+s} \left | \bigcap_{j=1}^k S_{i_j} \right| \right) \\
&= m^n - (n+s)m^{n-1} + \binom{n+s}{2} m^{n-2} - \sum_{1 \leq i_1 < \cdots < i_3 \leq n+s} \left | \bigcap_{j=1}^3 S_{i_j} \right| + \sum_{1 \leq i_1 < \cdots < i_4 \leq n+s} \left | \bigcap_{j=1}^4 S_{i_j} \right| \\
&- \sum_{1 \leq i_1 < \cdots < i_5 \leq n+s} \left | \bigcap_{j=1}^5 S_{i_j} \right| 
+\sum_{1 \leq i_1 < \cdots < i_6 \leq n+s} \left | \bigcap_{j=1}^6 S_{i_j} \right|   + \sum_{k=7}^{n+s} (-1)^{k} \left ( \sum_{1 \leq i_1 < \cdots < i_k \leq n+s} \left | \bigcap_{j=1}^k S_{i_j} \right| \right).  
\end{align*}
Then, Lemma~\ref{lem: three} yields:
\begin{align*}
&P_{DP}(M,\mathcal{H}) \\
&\geq m^n - (n+s)m^{n-1} + \binom{n+s}{2} m^{n-2} - \left(tm^{n-2} - x_\mathcal{H}m^{n-3} + |P_3|m^{n-3} \right)\\
& + |P_4|m^{n-3} - 2|P_4|x_\mathcal{H}m^{n-4}- \left(|P_5|m^{n-3} + \left(\binom{n+s}{5} - |P_5| \right) m^{n-4} \right) \\
&+|P_6|m^{n-3}- 2|P_6|x_\mathcal{H}m^{n-4}   - \sum_{k=7}^{n+s} \binom{n+s}{k} m^{n-4} \\
&= m^n - (n+s)m^{n-1} +  \left(\binom{n+s}{2} - t \right)m^{n-2} - a_3 m^{n-3} \\
&+ x_\mathcal{H}m^{n-3} - 2|P_4|x_\mathcal{H}m^{n-4} - 2|P_6|x_\mathcal{H}m^{n-4} - 2^{s}m^{n-4} \\
&\geq  m^n - (n+s)m^{n-1} +  \left(\binom{n+s}{2} - t \right)m^{n-2} - a_3 m^{n-3} \\
&+ m^{n-3} - (2(|P_4|+|P_6|+2^{s-1}))m^{n-4}.
\end{align*}
\end{proof}

We now prove Theorem~\ref{thm: cone}.

\begin{proof} 
We know that there must be $C, N_1 \in \N$ such that $P(M,m) \leq m^n - (n+s)m^{n-1} +  \left(\binom{n+s}{2} - t \right)m^{n-2} - a_3 m^{n-3} + Cm^{n-4}$ whenever $m \geq N_1$.  Also, when $m \geq 2(|P_4|+|P_6|)$ and $\mathcal{H} = (L,H)$ is an $m$-fold cover for $M$ with $x_\mathcal{H} > 0$,  Lemma~\ref{lem: lower} tells us $P_{DP}(M, \mathcal{H}) \geq m^n - (n+s)m^{n-1} +  \left(\binom{n+s}{2} - t \right)m^{n-2} - a_3 m^{n-3} + m^{n-3} - (2(|P_4|+|P_6|+2^{s-1}))m^{n-4}.$  Finally, there must be an $N_2 \in \N$ such that $m^{n-3} - (2(|P_4|+|P_6|+2^{s-1})+C)m^{n-4} \geq 0$ whenever $m \geq N_2$.  

Let $N = \max \{N_1, N_2, 2(|P_4|+|P_6|) \}$.  If $m \geq N$ and $\mathcal{H} = (L,H)$ is an $m$-fold cover for $M$ with $x_\mathcal{H} > 0$, then $P_{DP}(M, \mathcal{H}) - P(M,m) \geq m^{n-3} - (2(|P_4|+|P_6|+2^{s-1})+C)m^{n-4} \geq 0$.  Since we know that when $\mathcal{H} = (L,H)$ is an $m$-fold cover for $M$ with $x_\mathcal{H} = 0$, $P_{DP}(M, \mathcal{H})=P(M,m)$, we may conclude that $P_{DP}(M,m) = P(M,m)$ whenever $m \geq N$.      
\end{proof}

{\bf Acknowledgment.}  The authors would like to thank Hemanshu Kaul and Alexandr Kostochka for their guidance and encouragement.

\end{document}